\documentclass[11pt,reqno]{amsart}
\usepackage{amsmath}
\usepackage{amssymb}
\usepackage[mathscr]{eucal}
\usepackage[all]{xy}
\usepackage{mathrsfs}

\allowdisplaybreaks

\newtheorem{lemma}{Lemma}
\newtheorem{theorem}{Theorem}

\newtheorem{corollary}{Corollary}
\newtheorem{proposition}{Proposition}

\newtheorem{remark}{Remark}

\theoremstyle{definition}
\newtheorem{example}{Example}
\newtheorem{definition}{Definition}
\newtheorem{definition-lemma}{Definition-Lemma}
\newtheorem{definition-theorem}{Definition-Theorem}

\newtheorem*{ack}{Acknowledgements}
\newtheorem*{convention}{Convention}

\newcommand{\ov}{\overline}
\newcommand{\ot}{{\otimes}}

\newcommand{\Hom}{\mathrm{Hom}}
\newcommand{\Hoch}{\mathrm{Hoch}}
\newcommand{\Cycl}{\mathrm{Cycl}}
\newcommand{\Lag}{\mathrm{Lag}}
\newcommand{\Cont}{\mathrm{Cont}}

\renewcommand{\a}{\alpha}
\renewcommand{\b}{\beta}

\begin{document}

\title[Lie bialgebras and the cyclic homology of $A_\infty$ structures]{Lie bialgebras and the cyclic homology of\\ $A_\infty$ structures in topology}

\author{Xiaojun Chen}
\address{Department of Mathematics, University of Michigan, 530 Church Street, Ann Arbor, MI 48109, United States}
\email{xch@umich.edu}

\begin{abstract}
$A_\infty$ categories are a mathematical structure that appears in
topological field theory, string topology, and symplectic topology.
This paper studies the cyclic homology of a Calabi-Yau $A_\infty$
category, and shows that it is naturally an equivariant topological
conformal field theory, and in particular, contains an involutive
Lie bialgebra. Applications of the theory to string topology and the
Fukaya category are given; in particular, it is shown that there is
a Lie bialgebra homomorphism from the cyclic cohomology of the
Fukaya category of a symplectic manifold with contact type boundary
to the linearized contact homology of the boundary.
\end{abstract}


\maketitle

\tableofcontents

\section{Introduction}

In this article we study some algebraic structures on the cyclic
homology of an $A_\infty$ category that appears in topology.
$A_\infty$ structures have been widely studied in the past decades
by mathematicians from various fields, such as symplectic topology,
homological algebra, and string field theory. On the other hand, the cyclic
homology of algebras and categories has also played an important
role in the study of, for example, the non-commutative geometry and
$K$-theory. The motivation for us to study the cyclic homology of
$A_\infty$ structures is the recent work of Costello
(\cite{Costello2}), which asserts that:
\begin{itemize}
\item A Calabi-Yau $A_\infty$ category $\mathcal A$ is homotopy equivalent to an open topological conformal field theory (TCFT);
\item Associated to a Calabi-Yau $A_\infty$ category $\mathcal A$ is a natural open-closed TCFT;
\item The closed states of the associated open-closed TCFT is the Hochschild cohomology of $\mathcal A$.
\end{itemize}
Some similar work has also previously been obtained by Kontsevich,
part of which was summarized in \cite[\S11]{KS}. A Calabi-Yau
$A_\infty$ category is an $A_\infty$ category with a non-degenerate
pairing on the morphism spaces. Among the examples are the Fukaya
category of Lagrangian submanifolds of symplectic manifolds, and the
category of coherent sheaves on Calabi-Yau manifolds. Roughly
speaking, a {\it closed} TCFT is a functor from Segal's category to
the category of chain complexes, where Segal's category is the
category whose objects are the natural numbers (identified with
circles) and the morphisms between two sets of circles are (the
chain complex of) the moduli space of Riemann surfaces connecting
them. The open and open-closed TCFT's are defined similarly, where
the objects are the intervals and the unions of intervals and
circles respectively.

There is an $S^1$ action by rotation on the inputs and outputs of a
closed TCFT. In \cite{Getzler2} Getzler first studied the TCFT's
which are invariant under the $S^1$ action. We shall call these
TCFT's {\it equivariant} (with respect to the $S^1$ action). One of
the key observations is that Costello's functor from the open TCFT
to the closed TCFT respects the $S^1$ rotation. By modulo such a
rotation, one obtains on the one side, the cyclic cohomology of a
Calabi-Yau $A_\infty$ category; and on the other side, an
equivariant TCFT:

\begin{theorem}[Theorem~\ref{thm_ETCFT}]\label{main_thm}
Let $\mathcal A$ be a Calabi-Yau $A_\infty$ category. Then the cyclic cohomology of $\mathcal A$ has the structure of
an equivariant topological conformal field theory.
\end{theorem}

In  \cite{Getzler1,Getzler2} Getzler showed that a closed TCFT has
the structure of a Batalin-Vilkovisky algebra while an equivariant
TCFT has the structure of a {\it gravity algebra}. A gravity algebra
is a Lie algebra with more constraints, which he called the {\it
generalized Jacobi identities}. We shall give an explicit formula for
the Lie algebra of Getzler on the cyclic cohomology of a Calabi-Yau
$A_\infty$ category. In fact, we show  more, namely, the cyclic
cohomology is an {\it involutive Lie bialgebra} (see
Def.~\ref{def_Liebialg}):

\begin{theorem}[Theorem~\ref{thm_liebi}]\label{2nd_thm}
Let $\mathcal A$ be a Calabi-Yau $A_\infty$ category. Then the cyclic cochain complex of $\mathcal A$ has the structure of
a differential graded involutive Lie bialgebra.\end{theorem}

The construction of the Lie bialgebra is inspired by string topology (\cite{CS02}), 
and a similar result has been obtained by Barannikov (\cite[\S5]{Bara}) and Kontsevich-Soibelman (\cite[\S11]{KS}) in a slightly different and conceptual manner.
In string topology, Chas and Sullivan showed that the
equivariant homology of the free loop space of a compact manifold,
say $M$, also endows a Lie bialgebra structure. To compare these two
structures, recall that K.-T. Chen (\cite{Chen77}) and J. D. S.
Jones (\cite{Jo87}) showed that  the cyclic cohomology of the
cochain algebra of a simply connected $M$ is isomorphic to the
equivariant homology of its free loop space $LM$. However, a
transversality issue in string topology obstructs one to model the
Lie bialgebra of string topology on the cyclic cohomology of {\it
singular} or {\it cellular} cochain algebra of $M$. With the great
efficiency of rational homotopy theory, Hamilton-Lazarev (\cite{HL})
and Lambrechts-Stanley (\cite{LS}) independently construct a cyclic
$A_\infty$ algebra over $\mathbb Q$, which models the cochain
algebra of $M$ with Poincar\'e duality (in fact the one of
Lambrechts-Stanley is a differential graded algebra). Recall that a
cyclic $A_\infty$ algebra is a Calabi-Yau $A_\infty$ category with
one object, and one has the following:

\begin{theorem}[Theorem~\ref{thm_st}]\label{3rd_thm}
Suppose $M$ is a simply connected, compact manifold and $V$ is the cyclic $A_\infty$ algebra which models $M$, then
the Lie bialgebra on the cyclic
cohomology of $V$ models that of Chas and Sullivan in string topology.
\end{theorem}

Another application of Theorems \ref{main_thm} and \ref{2nd_thm} is
to symplectic topology. Recall that the Fukaya category of a compact
symplectic manifold (with or without boundary; for the former case,
some convexity conditions are required, for example, the boundary
is of contact type as in Example~\ref{ex_fs} in
\S2.1) is a Calabi-Yau $A_\infty$ category (\cite{Fuk02,Fuk09}). A
direct consequence of the above theorems is:

\begin{corollary} The cyclic cohomology of the Fukaya category of
a symplectic manifold is naturally an equivariant topological
conformal field theory, and in particular, it has the structure of
an involutive Lie bialgebra.
\end{corollary}

As an application to this corollary, we give a relationship between
the Fukaya category and symplectic field theory introduced by Eliashberg,
Givental and Hofer (\cite{EGH}).
Suppose $M$ is a symplectic manifold with contact type boundary $W$.
Mathematicians have obtained many interesting
structures for $M$ and $W$ in symplectic field theory; in particular, Cieliebak and
Latschev showed in \cite{CL} that,
the linearized contact homology of $W$ (a homology theory on
the space of Reeb orbits in $W$) endows the structure of an
involutive Lie bialgebra. Inspired by their work, we show that:

\begin{theorem}[Theorem~\ref{thm_liebimap}]\label{4th_thm}
Let $M$ be an exact symplectic manifold with contact type boundary $W$.
There is a chain map from the linearized contact chain complex of $W$
to the cyclic cochain complex of the Fukaya category
of $M$, which induces on homology a map of Lie bialgebras.
\end{theorem}

The paper is organized as follows: In Section 2, we recall the definitions of an $A_\infty$ category and its cyclic homology;
some examples arising from topology are given.
In Section 3 we give the definition of a Calabi-Yau category and discuss
its role in the study of topological conformal field theories. This proves Theorem~\ref{main_thm}.
In Section 4 we first prove Theorem~\ref{2nd_thm}.
Some relationship between the Lie bialgebras on the cyclic chain complex of an $A_\infty$ algebra
and in the non-commutative symplectic geometry is also discussed. Examples from string topology are given and
Theorem~\ref{3rd_thm} is shown.
In Section 5 we discuss some symplectic field theory and prove Theorem~\ref{4th_thm}.

\begin{ack}The current work is inspired by the talks given by Costello and Latschev at the CUNY Einstein Chair Seminar.
The author also benefits from interesting
conversations with Sullivan and Fukaya. He is very grateful to Yongbin Ruan
for many encouragements during the preparation of the paper.
\end{ack}

\section{Cyclic homology of $A_\infty$ categories}

\begin{definition}[$A_\infty$ category]Let $k$ be a field.
An $A_\infty$ category $\mathcal A$ consists of a set of objects $\mathcal Ob(\mathcal A)$,  a graded $k$-vector space
$\Hom(A_1,A_2)$ for each pair of objects $A_1,A_2\in\mathcal Ob(\mathcal A)$, and a sequence of
operators:
$$m_i: \Hom(A_0,A_1)\otimes \Hom(A_1, A_2)\otimes\cdots\otimes \Hom(A_{i-1},A_i)\to \Hom(A_0,A_i), $$
where $|m_i|=i-2$, for $ i=1,2,\cdots,$ satisfying the following $A_\infty$ relations:
$$
\sum_{p=0}^{i-1}\sum_{k=1}^{i-p}(-1)^{|a_1|+\cdots+|a_p|+p+\epsilon_{i-k+1}+\epsilon_k}m_{i-k+1}(a_1,\cdots, a_{p}, m_{k}(a_{p+1},\cdots,a_{p+k}),a_{p+k+1},\cdots,a_i)=0,
$$
where $$\epsilon_k=(k-1)|a_{p+1}|+(k-2)|a_{p+2}|+\cdots+|a_{p+k-1}|+\frac{k(k-1)}{2}$$ and
\begin{multline*}\epsilon_{i-k+1}=(i-k)|a_1|+\cdots+(i-k-p+1)|a_p|
+(i-k-p)(i-2+|a_{p+1}|+\cdots+|a_{p+k}|)\\
+(i-k-p-1)|a_{p+k+1}|+\cdots+|a_{i-1}|+\frac{(i-k)(i-k-1)}{2}.
\end{multline*}
\end{definition}

An $A_\infty$ category with one object is an $A_\infty$ algebra in
the usual sense, and conversely an $A_\infty$ category may be viewed
as an $A_\infty$ algebra ``with several objects". If an $A_\infty$
category $\mathcal A$ has vanishing $m_i$, for $i=3, 4,\cdots$, then
it is a differential graded (DG) category.

\begin{convention}\label{sign_conv} (1) The signs are always complicated in the definition of $A_\infty$ structures.
First, recall the Koszul sign convention:
Suppose $V$ is a graded space, and $\a,\b\in V$; then in any case when we switch $\a\otimes\b$ to $\b\otimes\a$,
there is a $(-1)^{|\a||\b|}$ popping out, i.e., $\a\otimes\b\mapsto (-1)^{|\a||\b|}\b\otimes\a$.

(2) Suppose $(V,\{m_i\})$ is an $A_\infty$ algebra.
Let $\Sigma V=V[1]$ be the desuspension of $V$, and denote by $T(\Sigma V)$ the cotensor algebra generated by $\Sigma V$.
Define
$$\overline{m}_i=\Sigma\circ m_i\circ (\Sigma^{-1})^{\otimes i}: (\Sigma V)^{\otimes i}\to\Sigma V,\,\, i=1,2,\cdots$$
i.e., if we denote the identity $id: V\to\Sigma V, a\mapsto \ov a$, then
$$\overline{m}_i(\ov a_1,\cdots,\ov a_i)=(-1)^{(i-1)|\ov a_1|+(i-2)|\ov a_{2}|+\cdots+|\ov a_{i-1}|}\Sigma(m_i(a_1,a_2,\cdots, a_i)).$$
Then $\overline m=\overline m_1+\overline m_2+\cdots$ defines a degree $-1$ co-differential on the coalgebra $T(\Sigma V)$. $(T(\Sigma V), \overline m)$ is
called the {\it bar construction} of $V$.
\end{convention}

\begin{definition}[Hochschild homology]
Let $\mathcal A$ be an $A_\infty$ category. The {\it Hochschild
chain complex} $\Hoch_*(\mathcal A)$ of $\mathcal A$ is the chain
complex whose underlying vector space is
$$
\bigoplus_{n=1}^{\infty}\bigoplus_{A_0,A_1,\cdots,A_n\in\mathcal Ob(\mathcal A)}
\Hom(A_0,A_1)\otimes\Hom(A_1, A_2)\otimes\cdots\otimes\Hom(A_n,A_0)
$$
with differential $b$ defined by
\begin{eqnarray}&&b(a_0,a_1,\cdots,a_n)\nonumber\\
&=&\sum_{j=0}^{n}\sum_{i=0}^{n-j}\pm(m_{i+j+1}(a_{n-j+1},\cdots,a_{n},a_0,\cdots,a_{i}),a_{i+1},\cdots,a_{n-j})\label{b_prime}\\
&+&\sum_{k=0}^{n-1}\sum_{i=1}^{n-k}\pm(a_0,a_1,\cdots,a_k,m_i(a_{k+1},\cdots,
a_{k+i}),a_{k+i+1},\cdots, a_{n})\label{b_prime'},
\end{eqnarray}
where the signs $\pm$ in (\ref{b_prime}) and (\ref{b_prime'}) are given below.
The associated homology is called the {\it Hochschild homology} of $\mathcal A$,
and is denoted by $HH_*(\mathcal A)$.
\end{definition}

With the sign convention discussed in the previous paragraph, the sign in (\ref{b_prime}) and (\ref{b_prime'}) are assigned as follows:
for simplicity, we assume $\mathcal A$ has one object $A$ (the general case is completely the same), and let $V=\Hom(A,A)$;
the Hochschild complex of $\mathcal A$ is nothing but the {\it twisted tensor product} of $V$ and $T(\Sigma V)$, from which the
sign follows. More precisely, for $(a_0,\ov a_1,\ov a_2,\cdots,\ov a_n)\in V\otimes T(\Sigma V)$,
\begin{eqnarray*}
&&b(a_0,\ov a_1,\ov a_2,\cdots,\ov a_n)\\
&=&\sum_{j}\sum_i(-1)^{\epsilon_1}(\Sigma^{-1}\circ\overline m_{i+j+1}(\ov a_{n-j+1},\cdots,\ov a_n,\ov a_0, \ov a_1,\cdots, \ov a_{i}),\ov a_{i+1},\cdots,\ov a_{n-j})\\
&+&\sum_{k}\sum_{i}(-1)^{\epsilon_2}(a_0,\ov a_1,\cdots, \ov a_{k}, \overline m_i(\ov a_{k+1},\cdots, \ov a_{k+i}),a_{k+i+1},\cdots, \ov a_n),
\end{eqnarray*}
where $\epsilon_1=(|\ov a_{n-j+1}|+\cdots+|\ov a_n|)(|\ov a_0|+\cdots+|\ov a_{n-j}|)$
and $\epsilon_2=|a_0|+|\ov a_1|+\cdots+|\ov a_{k}|$ (note $\epsilon_1$ and $\epsilon_2$ come from the Koszul sign convention).
Let $id$ be the identity map:
$$\begin{array}{cccl}
id:&V\otimes V\otimes\cdots\otimes V&\stackrel{id\otimes\Sigma^{\otimes n}}{\longrightarrow}&V\otimes\Sigma  V\otimes \cdots\otimes\Sigma V\\
&a_0\otimes a_1\otimes\cdots\otimes a_n&\longmapsto&(-1)^{n|a_0|+(n-1)|a_1|+\cdots +|a_{n-1}|} a_0\otimes\ov a_1\otimes\cdots\otimes\ov a_n.
\end{array}$$
The sign in (\ref{b_prime}) and (\ref{b_prime'}) are assigned such that the following diagram commutes:
$$\xymatrix{
V\otimes TV\ar[d]^b\ar[r]^-{id}&V\otimes T(\Sigma V)\ar[d]^b\\
V\otimes TV\ar[r]^-{id}&V\otimes T(\Sigma V).
}$$

\begin{remark}The definition of the Hochschild homology of an $A_\infty$ algebra was first introduced by
Getzler and Jones in {\cite{GJ}}. The signs in the above definition
are consistent with theirs (see \cite[Lemma 1.3 and Definition 3.5]{GJ}). In that paper the authors also introduced
the cyclic homology of an $A_\infty$ algebra, and hence the
cyclic {\it cohomology}
is also defined. While in the paper of Penkava and Schwarz~\cite{PS}, another version of cyclic cohomology is given.
An argument similar to Loday~\cite[Theorem 2.1.5]{Loday} shows that these two definitions coincide if the
ground field $k$ contains $\mathbb Q$. The following definition is due to Penkava and Schwarz.
\end{remark}

\begin{definition}[Cyclic homology]\label{def_cycl}
Suppose $\mathcal A$ is an $A_\infty$ category.
Let
\begin{multline*}t_{n}: \Hom(A_0,A_1)\otimes\cdots\otimes\Hom(A_{n-1},A_n)\otimes\Hom(A_n,A_0)\\
\to \Hom(A_n,A_0)\otimes\Hom(A_0,A_1)
\otimes\cdots\otimes\Hom(A_{n-1},A_n),
\end{multline*}
$n=0,1,2, \cdots$, be the linear map
\begin{equation}\label{def_t}t_n(a_0, a_1,\cdots, a_{n}):=(-1)^{n+|a_n|(|a_0|+\cdots+|a_{n-1}|)}(a_{n}, a_0,\cdots, a_{n-1}),
\end{equation}
where the factor $(-1)^n$ is the sign of the cyclic operator $t_n$. Extending $t_n$ to $\Hoch_*(\mathcal A)$ trivially, and let
$$t=t_0+t_1+t_2+\cdots,$$ the cokernel
$\Hoch_*(\mathcal A)/(1-t)$ of $1-t$ forms a chain complex with the induced differential from the Hochschild complex (still denoted by $b$).
Such chain complex is denoted by $\Cycl_*(\mathcal A)$, and is called the {\it Connes cyclic
complex} of $\mathcal A$. Its homology is called the {\it cyclic homology} of $\mathcal A$, and is denoted by $CH_*(\mathcal A)$.

The {\it cyclic cohomology} of $\mathcal A$ is the homology of the dual cochain complex $\Cycl^*(\mathcal A)$
of $\Cycl_*(\mathcal A)$. Namely, suppose $f\in\Hom(\Hoch_*(\mathcal A), k)$, then $f\in\Cycl^*(\mathcal A)$ if and only if for all
$\a\in\Hoch_*(\mathcal A)$, $f(\a)=f(t(\a))$.
\end{definition}

The well-definedness of the above definition is similar to the case of the cyclic homology of differential graded algebras (see, for example,
Loday~\cite[\S 2.1.4]{Loday}), the proof of
which is therefore omitted.
Like the Hochschild chain complex case, the signs in the cyclic chain complex are also delicate.
An elegant treatment of the signs, which is due to Quillen (\cite{Quillen}), greatly simplifies the issue.
As before, we assume the $A_\infty$ category $\mathcal A$ has one object $A$, and denote the morphism space by $V$.
Let $id$ be the identity map:
$$\begin{array}{cccl}
id:&V\otimes V\otimes\cdots\otimes V&\stackrel{\Sigma^{\otimes n+1}}{\longrightarrow}&\Sigma V\otimes\Sigma V\otimes \cdots\otimes\Sigma V\\
&a_0\otimes a_1\otimes\cdots\otimes a_n&\longmapsto&(-1)^{n|a_0|+(n-1)|a_1|+\cdots +|a_{n-1}|}\ov a_0\otimes\ov a_1\otimes\cdots\otimes\ov a_n,
\end{array}$$
and let $\overline t$ be the isomorphism
$$\begin{array}{cccl}
\overline t:&\Sigma V\otimes \Sigma V\otimes \cdots\otimes\Sigma V&\longrightarrow&\Sigma V\otimes\Sigma V\otimes\cdots\otimes\Sigma V\\
&\ov a_0\otimes\ov a_1\otimes \cdots\otimes\ov a_n&\longmapsto&(-1)^{|\ov a_n|(|\ov a_0|+\cdots+|\ov a_{n-1}|)}
\ov a_n\otimes\ov a_0\otimes\cdots\otimes\ov a_{n-1}.
\end{array}$$
Note that the sign in the image of $\overline t$ is only coming from the Koszul convention. We have the following:

\begin{lemma}Let $(V,\{m_i\})$ be an $A_\infty$ algebra, and let $\overline m_i$ and $id, t, \overline t$ be as above. The following diagram commutes:
$$\xymatrix{
TV\ar[d]^{id}\ar[rr]^-{t, m_i}&&TV\ar[d]^{id}\\
T(\Sigma V)\ar[rr]^-{\overline t,\overline m_i}&&T(\Sigma V).
}$$
\end{lemma}

The proof is a direct check. By the lemma one may identify the cyclic chain complex $\Cycl_*(V)$, with degree shifted down by 1, with
$T^+(\Sigma V)/(1-\overline t)$, where $T^+(\Sigma V)$ is the augmented ideal of $T(\Sigma V)$, i.e. $T^+(\Sigma V)=T(\Sigma V)/k$.
In other words, the cyclic complex of $V$, up to a degree shift, is isomorphic to space of cyclically coinvariant elements in the augmented bar construction.

For a general $A_\infty$ category $\mathcal A$, we apply the same
identification. The advantage is that we do not need to care about
the signs any more; all the signs, either coming from the $A_\infty$
structure or from the cyclic operator, just follows the Koszul sign
convention.

\subsection{Examples from topology}
The structure of $A_\infty$ algebras and categories naturally arise
in topology. Besides the classical example of $A_\infty$ spaces
discovered by Stasheff, there are two examples which have been
studied in recent years. One is the $A_\infty$ algebra model
of a simply-connected smooth manifold, the other is the $A_\infty$
category of Lagrangian submanifolds (the Fukaya category). In these
two cases the associated $A_\infty$ categories bear some additional
structure, which make them into the so-called {\it Calabi-Yau
$A_\infty$ categories}. We shall discuss the first example in \S4.3
and the second in \S5.

In the following we present two simplified examples. They are the
derived category of complex of coherent sheaves on a Fano toric
manifold and the
Fukaya-Seidel category of a symplectic Lefschtez fibration.
In both cases, some of the $A_\infty$ operators vanish, which
therefore greatly simplifies the computation. Despite of the
simplicity, they are proved to be an important role in the
understanding of the homological mirror symmetry conjecture
(\cite{AKO1,AKO2}).

\begin{example}[Derived category of coherent sheaves]

Recall that a {\it triangulated category} is an additive category
equipped with a shifting functor and a class of distinguished
triangles, satisfying some conditions (see, for example,
\cite[Chapter IV]{GM}). Suppose $\mathcal K$ is a triangulated
category. An ordered set of objects $E=(E_0,E_1,\cdots,E_n)$ is
called a {\it strong exceptional collection} if
\begin{equation}\label{stexco}
\Hom(E_j, E_i[k])=\left\{
\begin{array}{ll}
k\cdot id,&\mbox{if}\; i=j\;\mbox{and}\; k=0,\\
 0,&\mbox{if}\;j>i\;\mbox{or}\; k\ne 0.\end{array}
\right.
\end{equation}
A strong exceptional collection $E$ is called {\it full} if it generates
the category $\mathcal K$, that is, the smallest triangulated category that contains $E$
is $\mathcal K$ itself.

\begin{lemma}\label{cycl_ex}Suppose $\mathcal K$ is a triangulated category, and $(E_1,E_2,\cdots,E_n)$
is a strong exceptional collection. Denote by $B$ the endomorphism
algebra of
$$E_1\oplus E_2\oplus\cdots\oplus E_n;$$
then the cyclic homology $CH_*(B)$
is isomorphic to
$CH_*(k)^{\oplus n+1}$.
\end{lemma}

\begin{proof}Due to Equation (\ref{stexco}),
its cyclic chain complex
is exactly
$$\bigoplus_{i=0}^n\bigoplus_{j=1}^\infty
(\underbrace{\ov{id}_{E_i},\ov{id}_{E_i},\cdots,\ov{id}_{E_i}}_j)/(1-\ov t),
$$
and the lemma follows.
\end{proof}

One may view the set $\{E_0,E_1,\cdots,E_n\}$ as objects of a
category, where $\Hom(E_i,E_j)=\Hom_{\mathcal K}(E_i,E_j)$, then the
cyclic homology of $B$ is also the cyclic homology of that category.

The above lemma can be used to compute the cyclic homology of some
projective varieties. For example, for $\mathbb{CP}^n$, $(\mathscr
O_{\mathbb{CP}^n},\mathscr O_{\mathbb{CP}^n}(-1),\cdots, \mathscr
O_{\mathbb{CP}^n}(-n))$ forms a full strong exceptional collection
for the bounded derived category of coherent sheaves $\mathcal
D^b(\mathrm{coh}(\mathbb{CP}^n))$. Denote by $\mathscr E:= \mathscr
O_{\mathbb{CP}^n}\oplus\mathscr
O_{\mathbb{CP}^n}(-1)\oplus\cdots\oplus \mathscr
O_{\mathbb{CP}^n}(-n)$; Bondal and Kapranov showed in \cite{BK} that
there is an equivalence
$$\mathbf R\Hom(\mathscr E,-):\mathcal D^b(\mathrm{coh}(\mathbb{CP}^n))\to
\mathcal D^b(\mathrm{Mod}(B)),$$
where $\mathcal D^b(\mathrm{Mod}(B))$ is the derived category of complexes of right $B$-modules whose
cohomology is finitely generated.
Hence, by the above argument,
the cyclic homology of $\mathbb{CP}^n$ is isomorphic to
$CH_*(k)^{\oplus n+1}$. For some more examples, see~\cite{Kapranov,Keller2,AKO1,AKO2}.
\end{example}

\begin{example}[The Fukaya-Seidel category]\label{ex_fs}
We briefly recall for non-experts the Fukaya-Seidel category of symplectic Lefschetz fibrations.
The construction is due to P. Seidel (\cite{Seidel1,Seidel2}).

Recall that an {\it exact symplectic manifold with contact type boundary} is a
quadruple $(M,\omega,\eta,J)$, where $M$ is a compact $2n$ dimensional manifold with
boundary, $\omega$ is a symplectic 2-form on $M$, $\eta$ is a
1-form such that $\omega=d\eta$ and $J$ is a $\omega$-compatible
almost complex structure. These data also satisfy the
following two convexity conditions:
\begin{itemize}\item The negative Liouville
vector field defined by
$\omega(\cdot,X_\eta)=\eta$
points strictly inwards along the boundary of $M$;
\item The boundary of $M$ is weakly $J$-convex, which means that any
$J$-holomorphic curves cannot touch the boundary unless they are
completely contained in it.
\end{itemize}
An $n$ dimensional submanifold $L\subset M$ is
called {\it Lagrangian} if $\omega|L=0$. We always assume $L$ is
closed and is disjoint from the boundary of $M$. $L$ is called {\it
exact} if $\eta|L$ is an exact 1-form.

Take $k$ to be a field of characteristic 2. According to symplectic
geometry (see Fukaya~\cite{Fuk02} or Seidel~\cite{Seidel3}),
the exact Lagrangian submanifolds of $M$ form the objects of
an $A_\infty$ category, where the morphisms between
$L_i$ and $L_j$ are the Floer cochain complex
$CF^*(L_1,L_2;k)$, which is spanned by
the transversal intersection points of $L_1$ and $L_2$. The $A_\infty$ operators $\{m_i\}$
are given by counting the $J$-holomorphic disks. Namely,
for $(p_1,p_2,\cdots,p_i)\in\Hom(A_0,A_1)\otimes\cdots\otimes\Hom(A_{i-1},A_i)$,
$$m_i(p_1,p_2,\cdots,p_i)=\sum_{q\in\Hom(A_1,A_i)}\#\mathcal M(p_1,p_2,\cdots,p_i,q)\cdot q,$$
where $\#\mathcal M(p_1,p_2,\cdots,p_i,q)$ is the cardinality of the moduli space of $J$-holomorphic
disks whose boundary lies in $A_1,A_2,\cdots,A_i$ and goes through $p_1,p_2,\cdots,p_i,q$ in cyclic order.
If the moduli space is not zero dimensional, it is counted as 0.

The above $A_\infty$ category is called the {\it Fukaya category} of $M$, and is
denoted by $\mathcal Fuk(M)$. We shall discuss a little more about
it in the last section. However, from the definition, one sees that
the $A_\infty$ operators $m_i$ are in general non-trivial, and
therefore it is not easy to compute the Hochschild or cyclic
homology of $\mathcal Fuk(M)$.

Let $(M,\omega,\eta, J)$ be an exact symplectic manifold. A smooth function $f: M\to\mathbb C$ is called a
{\it symplectic Lefschetz fibration} if
$f$ have isolated non-degenerate critical points such that
near each of these critical points, say $p$, $f$ has the form
$$f(z_1,z_2,\cdots, z_n)=f(p)+z_1^2+z_2^2+\cdots+z_n^2,$$
and that the fibers of $f$ are symplectic submanifolds of $M$.
Assume that each critical value corresponds to one critical point.

Fix a regular value $\lambda_0$ of $f$, and consider an arc $\gamma\subset \mathbb C$ joining $\lambda_0$ to a critical value $\lambda_i=f(p_i)$.
For the symplectic Lefschetz fibration, there is a natural way of doing symplectic parallel transport along any path, and for $\gamma$
the set of points transported to $p_i$ forms a Lagrangian disk $\Delta_\gamma\subset M$. The boudary of $\Delta_\gamma|f^{-1}(\lambda_0)$
is therefore a Lagrangian sphere, which is called a {\it vanishing cycle}.
Let $\gamma_1,\gamma_2,\cdots,\gamma_r$ be a collection
of arcs in $\mathbb C$ joining $\lambda_0$ to the various critical values of $f$, and assume they
only intersect at $\lambda_0$, and are ordered in the clockwise direction around $p_0$. Each
arc $\gamma_i$ gives a Lagrangian sphere $L_i\subset f^{-1}(\lambda_0)$. After a small
perturbation we can assume that these spheres intersect each other
transversely inside $f^{-1}(\lambda_0)$.

\begin{definition-lemma}[Seidel]
Let $f:M\to\mathbb C$ be a symplectic Lefschetz fibration.
The directed $A_\infty$ category of vanishing cycles, denoted by
$\Lag_{\mathrm{vc}}^\to(f,\{\gamma_i\})$, is the $A_\infty$ category whose objects are
$L_1,L_2,\cdots,L_r$ corresponding to the vanishing cycles, and whose morphisms are given by
\begin{equation}\label{directed}\Hom(L_i, L_j)=\left\{\begin{array}{ll}CF^*(L_i,L_j),&\mbox{if}\; i<j,\\
k\cdot id,&\mbox{if}\; i=j,\\
0,&\mbox{if}\;i>j,
\end{array}\right.\end{equation}
with $m_k:\Hom(L_{i_0}, L_{i_1})\otimes\cdots\otimes\Hom(L_{i_{k-1}}, L_{i_k})\to \Hom(L_{i_0},L_{i_k})$
defined by
$$m_k=\left\{
\begin{array}{ll}
m_k|{{\mathcal Fuk(M)}},&\mbox{if}\; i_0<i_1<\cdots<i_k,\\
0,&\mbox{otherwise.}
\end{array}
\right.$$
\end{definition-lemma}

According to Seidel, any $A_\infty$ category $\mathcal A$ whose
morphisms between two objects are in the form (\ref{directed}) with
$CF^*(-,-)$ replaced by a finite dimensional vector space is called
{\it directed}. Similarly to Lemma~\ref{cycl_ex} we have the
following:

\begin{lemma}\label{cyclic_of_directed}
Suppose $\mathcal A$ is a directed $A_\infty$ category with $n$ objects, then
$$CH_*(\mathcal A)\cong CH_*(k)^{\oplus n}.$$
\end{lemma}

\begin{corollary}The cyclic
homology of the directed Fukaya category
$\Lag_{\mathrm{vc}}^\to(f,\{\gamma_i\})$ of vanishing cycles is
isomorphic as vector spaces to the direct sum of $r$ copies of
$CH_*(k)$:
$$CH_*(\Lag_{\mathrm{vc}}^\to(f,\{\gamma_i\}))\cong CH_*(k)^{\oplus r},$$
where $r$ is the number of critical points.
\end{corollary}

Seidel's construction of $\Lag_{\mathrm {vc}}^\to(f,\{\gamma_i\})$
depends on the order of the paths connecting the critical
points (hence the order of the vanishing cycles). A key theorem
that Seidel proves is that if we change the order of the vanishing
cycles, the two sets are related by a sequence of {\it mutations}. A
classical result of triangulated categories says that if two
exceptional collections are related by mutations, they generate the
the same triangulated category. Thus if we consider the derived
$A_\infty$ category $\mathcal
D(\Lag_{\mathrm{vc}}^\to(f,\{\gamma_i\}))$ (for a definition,
see~\cite{Seidel1}), where the vanishing cycles form a full strong
exceptional collection, then it is independent of the order, and
hence is an invariant of the Lefschetz fibration. The cyclic
homology of $\Lag_{\mathrm{vc}}^\to(f,\{\gamma_i\})$, which is
isomorphic to $CH_*(k)^{\oplus n}$ as a linear space, is independent
of the order of the vanishing cycles, too: a theorem of Keller
(\cite{Keller}) says that the cyclic homology is invariant under
derived equivalence.
\end{example}

\section{Calabi-Yau $A_\infty$ categories}

\begin{definition}[Calabi-Yau $A_\infty$ category]
An $A_\infty$ category $\mathcal A$ is called {\it Calabi-Yau} of degree $n$ if there is degree $n$ graded symmetric pairing
$$\Hom(B,A)\otimes\Hom(A,B)\stackrel{\langle\,,\,\rangle}{\longrightarrow}k$$
for each pair of objects $A,B$,
which is non-degenerate and cyclically invariant, where being {\it cyclically invariant} means
the multilinear map
$$
(a_0,a_1,\cdots,a_n)\mapsto(-1)^{(n-2)|a_0|}\langle a_0, m_n(a_1,\cdots, a_n)\rangle,
$$
where $a_i\in\Hom(A_{i},A_{i+1}), i=0, 1, \cdots, n$ (here and in
the following we identify $A_{n+1}$ with $A_0$), is invariant under
the cyclic permutation, i.e.
$$(-1)^{(n-2)|a_0|}\langle a_0, m_n(a_1,\cdots, a_n)\rangle=(-1)^{n+|a_n|(|a_0|+\cdots+|a_{n-1}|)}\cdot(-1)^{(n-2)|a_n|}\langle a_n,m_n(a_0,\cdots, a_{n-1})\rangle.$$
\end{definition}

From the definition, one sees that if a Calabi-Yau $A_\infty$
category has one object, then it is a cyclic $A_\infty$ algebra in
the sense of Getzler-Kapranov (\cite{GeK}). The directed Fukaya
category $\Lag_{\mathrm{vc}}^\to(f,\{\gamma_i\})$ of vanishing
cycles in a symplectic Lefschetz fibration in the previous section
is in general not Calabi-Yau.

\subsection{Topological conformal field theory}

According to Costello, Calabi-Yau $A_\infty$ categories play an
essential role in the understanding of topological conformal field
theories. In the following we briefly recall his results.

Suppose $S$ is a Riemann surface with boundary, and denote its boundary by $\partial S$.
Suppose there is an embedding of disjoint union of intervals and circles to $\partial S$.
Call the images of the intervals the {\it open boundaries} (denoted by $\{O\}$) and the
images of the circles the {\it closed boundaries} (denoted by $\{C\}$).
If the orientation of an open or closed boundary is compatible with the one induced from $S$, then we say such a boundary
is {\it incoming}; otherwise, it is {\it outgoing}.

Suppose $\Lambda$ is a set, called the set of D-branes.
We label the components of $\partial S\backslash(\{\mbox{open boundaries}\}\cup\{\mbox{closed boundaries}\})$ by the
elements
of $\Lambda$, the D-branes. Note for each open boundary $O$ of $S$, there are two D-branes associated to it,
one is labeled to the component containing the tail of $O$ (denoted by $t(O)$), and the other is
labeled to the component containing the head of $O$ (denoted by $s(O)$).

Riemann surfaces with open and closed boundaries and labelings form
a category, where the objects are disjoint unions of circles and
intervals together with labels, and the morphisms between two
objects are the Riemann surfaces connecting them. The composition of
two morphisms comes from the glueing of the outgoing boundaries of
the first Riemann surface with the incoming boundaries of the second
one. Note that when glueing the open boundaries, the corresponding
head and tail labelings should be the same. Denote the category by
$\mathcal M_{\Lambda}$, which is a symmetric monoidal category under
disjoint unions. Define
$$\mathscr{OC}_\Lambda:=C_*(\mathcal M_\Lambda),$$ which is again a symmetric monoidal category, where $C_*(-)$ is the
chain functor.
There are two full subcategories of $\mathscr{OC}_\Lambda$, denoted by $\mathscr O_\Lambda$ and $\mathscr C$ respectively,
one is the subcategory whose objects are purely open and the other is the one whose objects are all closed.

\begin{definition}[Topological conformal field theory]
\begin{enumerate}\item An {\it open-closed topological conformal field theory (TCFT)} is a h-split symmetric monoidal functor
from the category $\mathscr{OC}_\Lambda$ to the category of chain complexes;

\item An {\it open TCFT} is a h-split symmetric monoidal functor from $\mathscr O_\Lambda$ to the category of chain complexes;
and similarly,

\item A {\it closed TCFT} is a h-split symmetric monoidal functor from $\mathscr C$ to the category of chain complexes.
\end{enumerate}
\end{definition}

In the above definition, by a h-split (abbreviation of homologically
split) functor we mean that $F(\a)\otimes F(\b)\to F(\a\otimes\b)$
is split on the homology $HF(\a)\otimes
HF(\b)\stackrel{\cong}{\to}HF(\a\otimes\b)$. Denote by
$\mathrm{Kom}$ the category of chain complexes. There are inclusions
$\mathscr O_\Lambda\stackrel{i}{\to} \mathscr{OC}_\Lambda
\stackrel{j}{\leftarrow} \mathscr C$. The following is the main
result of Costello (\cite{Costello2}):

\begin{theorem}[Costello]\begin{enumerate}
\item
The category of open TCFTs with the set of D-branes $\Lambda$ is homotopy equivalent to
the category of Calabi-Yau $A_\infty$ categories with objects $\Lambda$;
\item
For any open TCFT $\Phi$, the homotopy universal functor $\mathbb Li_*\Phi: \mathscr{OC}\to\mathrm{Kom}$ is h-split, and hence
defines an open-closed TCFT;
\item
The closed states of the open-closed TCFT in $(2)$ is the Hochschild cohomology of the $A_\infty$ category of $(1)$.
\end{enumerate}
\end{theorem}

The proof of this theorem is heavily based on Costello's previous work \cite{Costello1} on the dual decomposition of the moduli spaces;
the reader may refer to \cite{Costello2} for more details. Before briefly going over the proof of Costello, let us note that
for $\mathscr C$, there is a free $S^1$ action on the incoming and outgoing boundaries. The following definition is due to Getzler (see \cite{Getzler2}).
To distinguish the $S^1$ invariant functor with the ordinary one as above, we here add an adjective {\it equivariant}):

\begin{definition}[Equivariant closed TCFT]
An equivariant closed TCFT is an h-split symmetric monoidal
functor from $\mathscr C$ to $\mathrm{Kom}$ which is invariant with respect to the $S^1$ action, that is,
a functor from $C_*^{S^1}(\mathcal M_{\mathrm{closed}})$ to $\mathrm{Kom}$, where $\mathcal M_{\mathrm{closed}}$
is the moduli space of Riemann surfaces with closed boundaries.
\end{definition}

\begin{theorem}\label{thm_ETCFT}
Let $\mathcal A$ be a Calabi-Yau $A_\infty$ category, then
the cyclic cochain complex $\Cycl^*(\mathcal A)$ defines an
equivariant TCFT.
\end{theorem}

\begin{proof}We follow closely Costello's original argument,
which goes roughly as follows: Suppose $A$ is a differential graded symmetric monoidal (DGMS) category, one may
view a functor $F:A\to\mathrm{Kom}$ as a left $A$ module. He argues that if in general $A\to B$ is a quasi-isomorphism of dgsm categoris,
then the category of modules over $A$ and the category of modules over $B$ modules are homotopy equivalent.
With this result, to study the modules over $A$, say,
it is sufficient to study the modules over $B$ which is homotopy equivalent to $A$.
This allows Costello to construct an explicit model of $\mathscr O_\Lambda$ and $\mathscr{OC}_\Lambda$.
Namely, he constructs a DGSM category which is quasi-isomorphic to the category $\mathscr O_\Lambda$
and then constructs a module over this category which is quasi-isomorphic to $\mathscr {OC}_\Lambda$.

More precisely, denote by $\mathcal N_{g,h,r,s}$ the moduli space of
Riemann surfaces of genus $g$, with $h$ boundary components, $r$ marked points
in the boundary and $s$ marked points in the interior. Here the $r$ marked points will later be viewed as inputs and outputs of open strings,
and the $s$ marked points as closed strings.
There is a partial compactification of $\mathcal N_{g,h,r,s}$ into an orbifold with corners, which parameterizes the Riemann surfaces as above
but possibly with nodes on the boundary. Inside
the compactification $\ov{\mathcal N}_{g,h,r,s}$ is an orbi-cell complex $D_{g,h,r,s}$, which parameterizes the Riemann surfaces
glued from discs, with each of its components has {\it at most} one interior marked point. We have (\cite{Costello1,Costello2}):
\begin{itemize}
\item The cellular chain complex of $D_{g,h,r,0}$ models $\mathscr O_{\Lambda}$. The discs with marked points
on its boundary is nothing but Stasheff's associahedra, from which one deduces that an open TCFT is homotopy
equivalent to a Calabi-Yau $A_\infty$
category;
\item The chain complex of $D_{g,h,r,s}$, $s\ge 0$, is a module (which is
almost free) over that of $D_{g,h,r,0}$. If we fatten the $s$ marked points into circles,
namely, consider the moduli space $G_{g,h,r,s}$,
where $G_{g,h,r,s}$ is the moduli space of Riemann surfaces glued from discs and annuli such that they are of genus $g$, with $h+s$ boundary
components, $r$ marked points in the first $h$ boundary components and one marked points in the each of the last $s$ boundary components
(coming from the annuli). Then
the chain complex of $G_{g,h,r,s}$ for all $g,h,r,s$, gives a model for $\mathscr{OC}_\Lambda$.
\item For an annulus with $r$ marked points in its outer boundary, say $a_1,a_2,\cdots,a_n$ (in cyclic order),
with D-brane labels $s(a_i)=\lambda_i, t(a_{i+1})=\lambda_{i+1}$, where $\lambda_{n+1}=\lambda_1$,
and a marked point $a$ in its inner boundary (viewed as the closed string) which is parallel to the arc between
$a_n$ and $a_1$,
if viewed as an element in the module of the chain complex of $D_{g,h,r,0}$, gives
a closed string state, which is exactly
$$\Hom(\lambda_1,\lambda_2)\otimes \Hom(\lambda_2,\lambda_3)\otimes\cdots\otimes \Hom(\lambda_n,\lambda_1).$$
Such an identity is compatible with the differential.
\end{itemize}
Basically this is Costello's proof; however, observe that there are $S^1$ ways to mark the point in
the inner boundary of the annulus, so if we modulo the $S^1$ action, namely, only consider
$D_{g,h,r,s}$ but not $G_{g,h,r,s}$ (this is Costello's original construction
in \cite{Costello1}), we get the closed state is exactly an element in the cyclic chain complex.
Also, we would prefer to view such annulus a map from the open string states to the closed string states, which then means
it is a cyclic cochain. This proves the theorem.
\end{proof}

\begin{remark} There are some technical details
in the proof of Costello, which we didn't mention, for example, the gradings. However,
according to Costello's compactification of the moduli spaces (\cite{Costello1}), the cyclic cohomology is probably
more natural in his argument than the Hochschild cohomology.
\end{remark}

\section{The Lie bialgebra structure}
In this section we construct the DG Lie bialgebra structure on the
cyclic cochain complex of a Calabi-Yau $A_\infty$ category $\mathcal A$.

First, we continue the discussion of the signs in \S2 by encoding them
with the non-degenerate pairing.
Note that the pairing $\langle\,,\,\rangle$ on $V$ induces a symplectic pairing on
$\Sigma V$: $\omega(\ov\a,\ov\b\rangle):=(-1)^{|\ov\a|}\langle\a,\b\rangle$.
The cyclic invariance of
$$a_0\otimes a_1\otimes\cdots\otimes a_n\mapsto(-1)^{(n-2)|a_0|}\langle a_0,m_n(a_1,a_2,\cdots,a_n)\rangle$$
implies the following commutative diagram:
$$\xymatrix{
\ov a_0\otimes\ov a_1\otimes\cdots\otimes\ov a_n\ar@{|->}[r]^-{\overline t}\ar@{|->}[d]
&(-1)^{|\ov a_n|(|\ov a_0|+\cdots+|\ov a_{n-1}|)}\ov a_n\otimes\ov a_0\otimes\cdots\otimes\ov a_{n-1}\ar@{|->}[d]\\
\omega(\ov a_0,\ov m_n(\ov a_1,\cdots,\ov a_n))\ar@{=}[r]
&(-1)^{|\ov a_n|(|\ov a_0|+\cdots+|\ov a_{n-1}|)}\omega(\ov a_n,\ov m_n(\ov a_0,\cdots,\ov a_{n-1})).
}$$
In other words, in the induced multilinear function
$$
\ov a_0\otimes\ov a_1\otimes\cdots\otimes\ov a_n\mapsto\omega(\ov a_0,\ov m_n(\ov a_1,\ov a_2,\cdots,\ov a_n)),
$$
the sign after the cyclic permutation
only comes from the Koszul convention, and the factor $(-1)^n$ from the cyclic permutation does not show up.

\subsection{Construction of the Lie bialgebra}

Suppose $\mathcal A$ is a Calabi-Yau $A_\infty$ category.

\begin{definition}[Lie bialgebra]\label{def_Liebialg} Let $L$ be a (possibly graded)
$k$-space. A Lie bialgebra on $L$ is the triple $(L, [\,,\,],\delta)$ such that
\begin{itemize}\item $(L, [\,,\,])$ is a Lie algebra;
\item $(L, \delta)$ is a Lie coalgebra;
\item The Lie algebra and coalgebra satisfy the following identity, called the {\it Drinfeld compatibility}:
$$\delta[a,b]=\sum_{(a)}((-1)^{|a''||b|}[a',b]\otimes a''+a'\otimes [a'', b])+\sum_{(b)}([a,b']\otimes b''+(-1)^{|a||b'|}b'\otimes[a,b'']),$$
for all $a,b\in L$, where we write $\delta(a)=\sum_{(a)}a'\otimes a''$ and $\delta(b)=\sum_{(b)}b'\otimes b''$.
\end{itemize}
If moreover, $[\,,\,]\circ\delta(a)\equiv 0$, for all $a\in L$,
$(L,[\,,\,],\delta)$ is said to be {\it involutive}. If the Lie
bracket has degree $k$ and the Lie cobracket has degree $l$, denote
the Lie bialgebra with degree $(l,k)$.
\end{definition}

\begin{definition-lemma}[Lie algebra]
Suppose $\mathcal A$ is a Calabi-Yau category of degree $n$.
Define $[\,,\,]:\Cycl^*(\mathcal A)\otimes\Cycl^*(\mathcal A)\to\Cycl^*(\mathcal A)$
by
\begin{eqnarray*}[f,g](\ov a_1,\ov a_2,\cdots,\ov a_{n})
&:=&\sum_{i<j}\sum_{p,q=1}^{\dim\Hom(A_i, A_{i+1})}\pm f(\ov a_1,\cdots,
\ov a_{i-1},\ov e_p,\ov a_{j+1},\cdots,\ov a_n)\cdot g(\omega^{pq}\ov f_q, \ov a_{i},\cdots, \ov a_{j})\\
&-&\sum_{i<j}\sum_{p,q=1}^{\dim\Hom(A_i, A_{i+1})}\pm g(\ov a_1,\cdots,\ov a_{i-1},\ov e_p,\ov a_{j+1},
\cdots,\ov a_n)\cdot f(\omega^{pq}\ov f_q, \ov a_{i},\cdots, \ov a_{j}),
\end{eqnarray*}
where $\{\ov e_p\}$ and $\{\ov f_q\}$ are bases of
$\Sigma\Hom(A_i,A_{j+1})$ and $\Sigma\Hom(A_{j+1},A_i)$
respectively, and $(\omega^{pq})=((-1)^{|e_p|+1}\langle
e_p,f_q\rangle)^{-1}$. Then $(\Cycl^*(\mathcal A),[\,,\,],b)$ forms
a differential graded Lie algebra of degree $2-n$.
\end{definition-lemma}

\begin{proof}

Recall the cyclic operator $\overline t: (\ov a_1,\ov a_2,\cdots,\ov a_n)\mapsto(\ov a_n,\ov a_1,\cdots, \ov a_{n-1})$;
if we set $N:=1+t+\cdots+t^n$, then $(1-t)N=N(1-t)=0$,
and we may write any homogeneous element in $\Cycl^*(\mathcal A)$ in the form
$\a=N(\ov a_1,\ov a_2,\cdots,\ov a_n)$, where $\ov a_i\in s(\Hom(A_i,A_{i+1})^*)$.
The bracket $[\,,\,]$ can
be written by the following formula: for $\a=N(\ov a_1,\ov a_2,\cdots,\ov a_n)$ and $\b=N(\ov b_1,\ov b_2,\cdots,\ov b_m)$,
$$[\a,\b]=\sum_{i=1}^n\sum_{j=1}^m\pm \omega(\ov a_i,\ov b_j)N(\ov a_1,\cdots,\ov a_{i-1}, \ov b_{j+1},
\cdots, \ov b_m,\ov b_1,\cdots, \ov b_{j-1},\ov a_{i+1},\cdots, \ov a_n),$$
where $\omega$ is the induced symplectic form on $s(\Hom(-,-)^*)$.
From the symplecticity of $\omega$, one sees that $[\,,\,]$ is graded skew symmetric.

We show the Jacobi identity:
for $\a=N(\ov a_1, \ov a_2,\cdots, \ov a_n),
\b=N(\ov b_1, \ov b_2,\cdots, \ov b_m),\gamma=N(\ov c_1, \ov c_2, \cdots, \ov c_p)$ in $\Cycl^*(\mathcal A)$,
\begin{eqnarray}
&&[[\a,\b],\gamma]\nonumber\\
&=&\sum_{i,j,k,l}\pm \omega(\ov a_i, \ov b_j)\omega(\ov a_k, \ov c_l)N(\ov a_1,\cdots, \ov b_{j+1}, \cdots, \ov b_{j-1},\cdots, \ov c_{l+1},\cdots,\ov c_{l-1},\cdots, \ov a_n)\label{jacobi1}\\
&+&\sum_{i,j,k,l}\pm \omega(\ov a_i, \ov b_j)\omega(\ov b_k, \ov c_l)N(\ov a_1,\cdots, \ov b_{j+1},\cdots, \ov c_{l+1},\cdots, \ov c_{l-1},\cdots, \ov b_{j-1}, \cdots, \ov a_n).\label{jacobi2}
\end{eqnarray}
Similarly, we have
\begin{eqnarray}
&&[[\b,\gamma],\a]\nonumber\\
&=&\sum_{i,j,k,l}\pm \omega(\ov b_j, \ov c_l)\omega(\ov b_k,\ov a_i)N(\ov b_1,\cdots,\ov c_{l+1},\cdots, \ov c_{l-1},\cdots, \ov a_{i+1}, \cdots, \ov a_{i-1},\cdots, \ov b_m)\label{jacobi3}\\
&+&\sum_{i,j,k,l}\pm \omega(\ov b_j, \ov c_l)\omega(\ov c_k,\ov a_i)N(\ov b_1,\cdots, \ov c_{l+1},\cdots, \ov a_{i+1},\cdots, \ov a_{i-1},\cdots, \ov c_{l-1},\cdots, \ov b_m),\label{jacobi4}
\end{eqnarray}
and
\begin{eqnarray}
&&[[\a,\b],\gamma]\nonumber\\
&=&\sum_{i,j,k,l}\pm \omega(\ov c_l, \ov a_i)\omega(\ov c_k, \ov b_j)N(\ov c_1,\cdots, \ov a_{i+1},\cdots, \ov a_{i-1},\cdots, \ov b_{j+1}, \cdots,\ov b_{j-1},\cdots, \ov c_p)\label{jacobi5}\\
&+&\sum_{i,j,k,l}\pm \omega(\ov c_l,\ov a_i)\omega(\ov a_k, \ov b_j)N(\ov c_1, \cdots, \ov a_{i+1},\cdots, \ov b_{j+1},\cdots, \ov b_{j-1},\cdots, \ov a_{i-1},\cdots, \ov c_p).\label{jacobi6}
\end{eqnarray}
By the cyclic invariance of $N$ and symplecticity of $\omega$, (\ref{jacobi1}) cancels with (\ref{jacobi6}),
so do (\ref{jacobi2}) with (\ref{jacobi3}) and (\ref{jacobi4}) with (\ref{jacobi5}).
This proves the Jacobi identity.

Finally, we show that the bracket commutes with the boundary:
\begin{eqnarray}(b[f,g])(\ov a_1,\ov a_2,\cdots,\ov a_{n})
&=&[f,g](b(\ov a_1,\ov a_2,\cdots, \ov a_{n}))\nonumber\\
&=&[f,g]\big(\sum_{i}\sum_{k}\pm(\ov a_1,\cdots,\overline m_k(\ov a_{i},\cdots,\ov a_{i+k-1}),\cdots,\ov a_n)\big)\label{Eq_1'}\\
&&+[f,g]\big(\sum_{j}\sum_{k}\pm(\overline m_k(\ov
a_{n-j},\cdots,\ov a_n,\ov a_1,\cdots,\ov a_{i}),\cdots,\ov
a_{n-j-1})\big),\label{Eq_1''}
\end{eqnarray}
while
\begin{eqnarray}
&&([bf,g]+(-1)^{|f|}[f, bg])(\ov a_1,\ov a_2,\cdots,\ov a_{n})\nonumber\\
&=&\sum_{i,j}\sum_{p,q}\pm f(b(\ov a_i,\cdots,\ov a_j,\ov e_p))\cdot g(\omega^{pq}\ov e_q,\ov a_{j+1},\cdots,\ov a_{i-1})\label{Eq_2}\\
&+&\sum_{i,j}\sum_{p,q}\pm f(\ov a_i,\cdots,\ov a_j, \ov e_p)\cdot g(b(\omega^{pq}\ov e_q,\ov a_{j+1},\cdots,\ov a_{i-1})).\label{Eq_3}
\end{eqnarray}
From the definition of $[\,,\,]$, one sees $(\ref{Eq_2})+(\ref{Eq_3})$ contains more terms than $(\ref{Eq_1'})+(\ref{Eq_1''})$, namely, those
terms involving $\overline m_k$ acting on $\ov e_p$. The extra terms coming from
$(\ref{Eq_2})$ are
\begin{equation}
\sum_{p,q}\sum_r f(\ov a_{l+1},\cdots, \overline m_r(\ov a_k,\cdots, \ov a_j, \ov e_p, \ov a_{i},\cdots,\ov a_l))\cdot g(\omega^{pq}\ov e_q,\ov a_{j+1},\cdots, \ov a_{i-1})\label{Eq_4}
\end{equation}
and the ones from $(\ref{Eq_3})$ are
\begin{equation}
\sum_{p,q}\sum_r(\ov a_i,\cdots, \ov a_j,\ov e_p)\cdot g(\overline m_r(\ov a_k,\cdots, \ov a_{i-1}, \omega^{pq}\ov e_q,\ov a_{j+1},\cdots,\ov a_l),\cdots,\ov a_{k-1}).\label{Eq_5}
\end{equation}
However, these two groups of terms cancel with each other: by the non-degeneracy of the pairing,
\begin{eqnarray*}&&\sum_{p,q}\overline m_r(\ov a_k, \cdots,\ov a_j,\ov e_p, \ov a_i, \cdots,\ov a_l)\otimes\omega^{pq}\ov e_q\\
&=&\sum_{p,q}\sum_{s,t}\langle \overline m_r(\ov a_k, \cdots,\ov a_j,\ov e_p,\ov a_i, \cdots,\ov a_l), \ov e_s\rangle \omega^{st}\ov e_t\otimes\omega^{pq}\ov e_q\\
&=&\sum_{p,q}\sum_{s,t}\ov e_t\otimes \omega^{st}\cdot\langle \overline m_r(\ov a_k, \cdots,\ov a_j,\ov e_p,\ov a_i, \cdots,\ov a_l),\ov e_s\rangle \omega^{pq}\ov e_q\\
&\stackrel{\rm{cyclicity}}{=}&\sum_{p,q}\sum_{s,t}\pm\ov e_t\otimes \omega^{st}\cdot\langle\overline m_r(\ov a_i, \cdots, \ov a_l, \ov e_s,\ov a_k, \cdots,\ov a_j), \ov e_p\rangle \omega^{pq}\ov e_q\\
&=&\sum_{s,t}\pm\ov e_t\otimes \overline m_r(\ov a_i,\cdots,\ov a_l,\omega^{st}\ov e_s,\ov a_k,\cdots,\ov a_j)\\
&=&-\sum_{s,t}\pm\ov e_t\otimes \overline m_r(\ov a_i,\cdots,\ov a_l,\omega^{ts}\ov e_s,\ov a_k,\cdots,\ov a_j).
\end{eqnarray*}
By substituting the above identity into (\ref{Eq_4}) we get exactly (\ref{Eq_5}), and therefore $(\Cycl^*(\mathcal A),[\,,\,], b)$
is a differential graded Lie algebra.
\end{proof}

\begin{definition-lemma}[Lie coalgebra]
Suppose $\mathcal A$ is a Calabi-Yau category of degree $n$.
Define $\Cycl^*(\mathcal A)\to
\Cycl^*(\mathcal A)\otimes\Cycl^*(\mathcal A)$
by
\begin{eqnarray*}
&&(\delta f)(\ov a_1, \ov a_2,\cdots,\ov a_{n})\otimes(\ov b_1,\ov b_2,\cdots, \ov b_{m})\\
&:=&\sum_{i=1}^n\sum_{j=1}^m\sum_{p,q=1}^{\dim\Hom(A_i, B_j)}\pm f(\ov a_1,\cdots,\ov a_{i-1},\ov e_p,\ov b_j,\cdots, \ov b_{m},\ov b_1,\cdots,\ov b_{j-1}, \omega^{pq}
\ov f_q,\ov a_i,\cdots, \ov a_{n}),
\end{eqnarray*}
where $\{\ov e_p\}$ and $\{\ov f_q\}$ are bases of $\Sigma\Hom(A_i,
B_j)$ and $\Sigma\Hom(B_j,A_i)$ respectively, and
$(\omega^{pq})=((-1)^{|e_p|+1}\langle e_p,f_q\rangle)^{-1}$. Then
$(\Cycl^*(\mathcal A),\delta,b)$ forms a Lie coalgebra of degree
$2-n$.
\end{definition-lemma}

\begin{proof} From the definition of $\delta$, the following two statements are obvious:\begin{enumerate}
\item[(1)] $\delta f$ is well defined, namely, the value of $\delta f$ is invariant under the cyclic permutations of
$(\ov a_1,\ov a_2,\cdots,\ov a_n)$ and $(\ov b_1,\ov b_2,\cdots,\ov b_m)$;
\item[(2)] $\delta f$ is (graded) skew-symmetric, namely, if we switch $(\ov a_1,\ov a_2,\cdots,\ov a_n)$ and $(\ov b_1,\ov b_2,\cdots,\ov b_m)$,
the sign of the value of $\delta f$ changes.
\end{enumerate}
If writing $f\in\Cycl^*(V)$ in the form $f=N(\ov a_1,\ov a_2,\cdots,\ov a_n)$, then
\begin{eqnarray}
\delta f&=&\delta N(\ov a_1,\ov a_2,\cdots,\ov a_n)\nonumber\\
&=&\sum_{i<j}\omega(\ov a_i,\ov a_j)N(\ov a_1,\cdots,\ov a_{i-1},\ov a_{j+1},\cdots,\ov a_n)\otimes N(\ov a_{i+1},\cdots,\ov a_{j-1})\label{co_jacobi1}\\
&-&\sum_{i<j}\omega(\ov a_i,\ov a_j)N(\ov a_{i+1},\cdots,\ov a_{j-1})\otimes N(\ov a_1,\cdots,\ov a_{i-1},\ov a_{j+1},\cdots,\ov a_n)\label{co_jacobi2}.
\end{eqnarray}
Therefore
\begin{eqnarray}
&&(id\otimes\delta)\circ\delta N(\ov a_1,\ov a_2,\cdots,\ov a_n)\nonumber\\
&=&\sum_{i<j}\sum_{k<l}\omega(\ov a_i,\ov a_j)\omega(\ov a_k,\ov a_l)
  N(\ov a_1,\cdots,\ov a_{i-1},\ov a_{j+1},\cdots,\ov a_n)\nonumber\\
  &&\quad\quad\quad\quad\quad\quad\quad\quad\otimes
   N(\ov a_{j+1},\cdots,\ov a_{k-1},\ov a_{l+1},\cdots,\ov a_{j-1})\otimes N(\ov a_{k+1},\cdots,\ov a_{l-1})\label{cj1}\\
& -&\sum_{i<j}\sum_{k<l}\omega(\ov a_i,\ov a_j)\omega(\ov a_k,\ov a_l)
  N(\ov a_1,\cdots,\ov a_{i-1},\ov a_{j+1},\cdots,\ov a_n)\otimes N(\ov a_{k+1},\cdots,\ov a_{l-1})\nonumber\\
  &&\quad\quad\quad\quad\quad\quad\quad\quad\otimes N(\ov a_{j+1},\cdots,\ov a_{k-1},\ov a_{l+1},\cdots,\ov a_{j-1})\label{cj2}\\
&+&\sum_{i<j}\sum_{k<l}\omega(\ov a_i,\ov a_j)\omega(\ov a_k,\ov a_l)
  N(\ov a_{i+1},\cdots,\ov a_{j-1})\otimes N(\ov a_1,\cdots,\ov a_{k-1},\cdots,\ov a_{l+1},\cdots,\ov a_n)\nonumber\\
  &&\quad\quad\quad\quad\quad\quad\quad\quad\otimes N(\ov a_{k+1},\cdots,\ov a_{i-1},\ov a_{j+1},\cdots,\ov a_{l-1})\label{cj3}\\
&-&\sum_{i<j}\sum_{k<l} \omega(\ov a_i,\ov a_j)\omega(\ov a_k,\ov a_l)
 N(\ov a_{i+1},\cdots,\ov a_{j-1})\otimes N(\ov a_{k+1},\cdots,\ov a_{i-1},\ov a_{j+1},\cdots,\ov a_{l-1})\nonumber\\
 &&\quad\quad\quad\quad\quad\quad\quad\quad\otimes N(\ov a_1,\cdots,\ov a_{k-1},\cdots,\ov a_{l+1},\cdots,\ov a_n)\label{cj4}\\
&+&\sum_{i<j}\sum_{k<l}\omega(\ov a_i,\ov a_j)\omega(\ov a_k,\ov a_l)
 N(\ov a_{i+1},\cdots,\ov a_{j-1})\nonumber\\
 &&\quad\quad\quad\quad\quad\quad\quad\quad
 \otimes N(\ov a_1,\cdots,\ov a_{i-1},\ov a_{j+1},\cdots,\ov a_{k-1},\ov a_{l+1},\cdots,\ov a_n)\otimes N(\ov a_{k+1},\cdots,\ov a_{l-1})\label{cj5}\\
&-&\sum_{i<j}\sum_{k<l}\omega(\ov a_i,\ov a_j)\omega(\ov a_k,\ov a_l)
 N(\ov a_{i+1},\cdots,\ov a_{j-1})\nonumber\\
 &&\quad\quad\quad\quad\quad\quad\quad\quad
 \otimes N(\ov a_{k+1},\cdots,\ov a_{l-1})\otimes N(\ov a_1,\cdots,\ov a_{i-1},\ov a_{j+1},\cdots,\ov a_{k-1},\ov a_{l+1},\cdots,\ov a_n),\label{cj6}
\end{eqnarray}
where $(\ref{cj1})+(\ref{cj2})$ come from $(\ref{co_jacobi1})$ and the other terms from $(\ref{co_jacobi2})$.

Let $\tau:\a\otimes\b\otimes\gamma\mapsto\pm\gamma\otimes\a\otimes\b$ be the cyclic permutation of three elements,
then
\begin{align*}
(\ref{cj1})+\tau^2(\ref{cj3})&=0,&\tau (\ref{cj1})+(\ref{cj3})&=0,&\tau^2(\ref{cj1})+\tau(\ref{cj3})&=0;\\
(\ref{cj2})+\tau(\ref{cj4})&=0,&\tau (\ref{cj2})+\tau^2(\ref{cj4})&=0,&\tau^2(\ref{cj2})+(\ref{cj4})&=0;\\
(\ref{cj5})+\tau^2(\ref{cj6})&=0,&\tau (\ref{cj1})+(\ref{cj6})&=0,&\tau^2(\ref{cj5})+\tau(\ref{cj6})&=0;
\end{align*}
that means, $(\tau^2+\tau+id)\circ(id\otimes\delta)\circ\delta f=0$. This proves the co-Jacobi identity.

Next, we show that $b$ respects the cobracket. This is similar to the Lie case. By definition,
\begin{eqnarray}
& &((b\otimes id\pm id\otimes b)\delta (f))((\ov a_1, \ov a_2,\cdots,\ov a_{n})\otimes(\ov b_1,\ov b_2,\cdots, \ov b_{m}))\nonumber\\
&=&\delta f(b(\ov a_1, \ov a_2,\cdots,\ov a_{n})\otimes(\ov b_1,\ov b_2,\cdots, \ov b_{m})\pm(\ov a_1, \ov a_2,\cdots,\ov a_{n})\otimes b(\ov b_1,\ov b_2,\cdots, \ov b_{m})),\label{codiff1}
\end{eqnarray}
while
\begin{eqnarray}
&&\delta(b(f))((\ov a_1, \ov a_2,\cdots,\ov a_{n})\otimes(\ov b_1,\ov b_2,\cdots, \ov b_{m}))\nonumber\\
&=&\sum_{i=1}^n\sum_{j=1}^m\sum_{p,q}\pm b(f)(\ov a_1,\cdots,\ov a_{i-1},\ov e_p,\ov b_j,\cdots, \ov b_{m},\ov b_1,\cdots,\ov b_{j-1}, \omega^{pq}\ov e_q,\ov a_i,\cdots, \ov a_{n})\label{codiff3}\\
&-&\sum_{j=1}^m\sum_{i=1}^n\sum_{p,q}\pm b(f)(\ov b_1,\cdots,\ov b_{j-1},\ov e_p,\ov a_i,\cdots, \ov a_{n},\ov a_1,\cdots,\ov a_{i-1},\omega^{pq}\ov e_q,\ov b_j,\cdots, \ov b_{m}).\label{codiff4}
\end{eqnarray}
Compared with $(\ref{codiff1})$, $(\ref{codiff3})+(\ref{codiff4})$ has more terms
\begin{eqnarray}
&&\sum_{p,q}f(\ov a_1,\cdots, \overline m_r(\ov a_{k},\cdots, \ov e_p,\cdots,\ov b_{l-1}),\ov b_l,\cdots, \omega^{pq}\ov e_q,\cdots, \ov a_n)\label{cob1}\\
&+&\sum_{p,q}f(\ov a_1,\cdots,\ov e_p,\cdots,\ov b_{j-1},\ov m_r(\ov b_j,\cdots, \omega^{pq}\ov e_q,\cdots,\ov a_l),\cdots,\ov a_n)\label{cob2}\\
&+&\sum_{p,q}f(\ov a_1,\cdots,\ov m_r(\ov a_k,\cdots,\ov e_p,\ov b_j,\cdots,\omega^{pq}\ov e_q,\cdots, \ov a_{l-1}),\ov a_{l},\cdots,\ov a_n),\label{cob3}
\end{eqnarray}
However, a similar argument as in the Lie case, $(\ref{cob1})+(\ref{cob2})$ vanishes, and the terms in
$(\ref{cob3})$ come in pair (counting $e_p$ and $e_q$), which
cancel within themselves. This proves that the differential commutes with the cobracket.
\end{proof}

\begin{theorem}\label{thm_liebi}
Let $\mathcal A$ be a Calabi-Yau $A_\infty$ category of degree $n$.
The cyclic cochain complex $\Cycl^*(\mathcal A)$ has the structure
of a DG involutive Lie bialgebra of degree $(2-n,2-n)$.
\end{theorem}

\begin{proof}
From the above three definitions, we only need to prove that $(\Cycl^*(\mathcal A), [\,,\,],\delta)$ satisfies the Drinfeld compatibility and
the involutivity.
Let $\a=N(\ov a_1, \cdots, \ov a_n)$ and $\b=N(\ov b_1, \cdots, \ov b_m)$, and write
$\delta(\a)=\a^{(1)} \ot \a^{(2)}$ and $\delta(\b)=\b^{(1)}\ot
\b^{(2)}$.

We have
$$[\a,\b]
= \sum_{i,j} \pm \omega(\ov a_i,\ov b_j)
N(\ov a_{i+1},\cdots,\ov a_,\ov a_1\cdots, \ov a_{i-1}, \ov b_{j+1},\cdots,\ov b_m,\ov b_1,\cdots,\ov b_{j-1})$$ and
\begin{align*}
& \delta[\a,\b] \\
=&\sum_{i,j,k,l} \pm\omega( \ov a_i,\ov b_j)\omega( \ov a_k,\ov a_l)
N(\ov a_{k+1},\cdots,\ov a_{l-1}) \ot N(\ov a_{l+1},\cdots,\ov a_{i-1},\ov b_{j+1},\cdots,\ov b_{j-1},\ov a_{i+1},\cdots,\ov a_{k-1})\\
+&   \sum_{i,j,k,l} \pm\omega( \ov a_i,\ov b_j)\omega( \ov a_k,\ov a_l)
N(\ov a_{k+1},\cdots,\ov a_{i-1},\ov b_{j+1},\cdots,\ov b_{l-1}) \ot
N(\ov b_{l+1},\cdots,\ov b_{j-1},\ov a_{i+1},\cdots,\ov a_{k-1}) \\
+ & \sum_{i,j,k,l} \pm\omega( \ov a_i,\ov b_j)\omega( \ov a_k,\ov a_l)
N(\ov a_{k+1},\cdots,\ov a_{i-1},\ov b_{j+1},\cdots,\ov b_{j-1},\ov a_{i+1},\cdots,\ov a_{l-1})
\ot
N(\ov a_{l+1},\cdots,\ov a_{k-1}) \\
+ & \sum_{i,j,k,l} \pm\omega( \ov a_i,\ov b_j)\omega( \ov a_k,\ov a_l)
N(\ov b_{k+1},\cdots,\ov b_{l-1}) \ot
N(\ov b_{l+1},\cdots,\ov b_{j-1},\ov a_{i+1},\cdots,\ov a_{i-1},\ov b_{j+1},\cdots,\ov b_{k-1}) \\
+ & \sum_{i,j,k,l} \pm\omega( \ov a_i,\ov b_j)\omega( \ov a_k,\ov a_l)
N(\ov b_{k+1},\cdots,\ov b_{j-1},\ov a_{i+1},\cdots,\ov a_{l-1}) \ot
N(\ov a_{l+1},\cdots,\ov a_{i-1},\ov b_{j+1},\cdots, \ov b_{k-1}) \\
+ &\sum_{i,j,k,l} \pm\omega( \ov a_i,\ov b_j)\omega( \ov a_k,\ov
a_l) N(\ov b_{k+1},\cdots,\ov b_{j-1},\ov a_{i+1},\cdots,\ov
a_{i-1},\ov b_{j+1},\cdots,\ov b_{l-1}) \ot N(\ov b_{l+1},\cdots,\ov
b_{k-1}).
\end{align*}
In the above, the second summation and the fifth summation cancel
with each other. The first summation is equal to $\a^{(1)} \ot
[\a^{(2)}, \b]$; the third summation is equal to $[\a^{(1)}, \b] \ot
\a^{(2)}$; the forth summation is equal to $\b^{(1)} \ot [\a,
\b^{(2)}]$; and the sixth summation is equal to $[\a, \b^{(1)}] \ot
\b^{(2)}$. Thus we obtain the Drinfeld compatibility.

Let $\a=N(\ov a_1,\cdots,\ov a_n)$, then
\begin{eqnarray*}
\delta(\a)&=&\sum_{i<j}\pm\omega(\ov a_i,\ov a_j)N(\ov a_1,\cdots,\ov a_{i-1},\ov a_{j+1},\cdots,\ov a_n)\otimes N(\ov a_{i+1},\cdots,\ov a_{j-1})\\
&-&\sum_{i<j}\pm\omega(\ov a_i,\ov a_j)N(\ov a_{i+1},\cdots,\ov a_{j-1})\otimes N(\ov a_1,\cdots,\ov a_{i-1},\ov a_{j+1},\cdots,\ov a_n).
\end{eqnarray*}
Noting that the two terms in the right hand side are anti-symmetric and by a similar argument as above, one obtains that
$[\,,\,]\circ\delta=0$ holds identically.
\end{proof}

Geometrically, for a Calabi-Yau $A_\infty$ category $\mathcal A$, if
we view $\mathcal A$ as an open TCFT, then the bracket and cobracket
are given by the gluing and splitting of the open states of
$\mathcal A$ in all possible ways. The construction is inspired from
string topology (see \cite{CS02} and also \S\ref{ex_st}), and some
of the above computations are similar to \cite{CEG}.

\subsection{Noncommutative symplectic geometry}

Our construction of the Lie bialgebra was inspired by string
topology (see \cite{CS02,SullivanSurvey} and \S\ref{ex_st}).
After the first version of the paper was put on arXiv, we were
informed by Kontsevich that he and Soibelman (\cite[\S10]{KS}) have
a nice cohomological characterization of the cyclic structure, and
our formula can also be deduced from theirs. Some related results
are Kontsevich \cite{Kontsevich1}, Ginzburg \cite{Ginzburg},
Schedler \cite{Schedler} and Hamilton \cite{Ham}; an essentially equivalent result
also appears in Barannikov \cite{Bara}. In this
subsection we would like to study this relationship in more detail.

\begin{definition}[Noncommutative differential and de Rham forms]
Let $V$ be a vector space. The module of noncommutative $1$-forms $\Omega^1(V)$ is defined by
$$\Omega^1(V):=T(V^*)\oplus T^+(V^*).$$
There is a $T(V^*)$-bimodule structure on $\Omega^1(V)$ which is given by
$$a\cdot(x\otimes y):=ax\otimes y;\quad (x\otimes y)\cdot a:=x\otimes ya-xy\otimes a,$$
for $a,x\in T(V^*)$ and $y\in T^+(V^*)$. Define $d: T(V^*)\to\Omega^1(V)$ by
$$d(x):=1\otimes x, \quad d(1)=0.$$
Let $A:=T(V^*)$. The algebra of {\it noncommutative differential forms} $\Omega^\bullet(V)$ are defined by
$$\Omega^\bullet(V):=T_{A}(\Sigma\Omega^1(V))=A\oplus\bigoplus_{n=1}^\infty\underbrace{\Sigma\Omega^1(V)\otimes_A\cdots\otimes_A\Sigma\Omega^1(V)}_{n}.$$
Lifting $d$ to $\Omega^\bullet(V)$ gives a differential graded associative algebra structure on $\Omega^\bullet(V)$.
The {\it noncommutative de Rham forms} $DR^\bullet(V)$ are defined by
$$DR^\bullet(V):=\frac{\Omega^\bullet(V)}{[\Omega^\bullet(V),\Omega^\bullet(V)]},$$
with the induced differential, still denoted by $d$.
\end{definition}

\begin{definition}[Lie derivative and contraction]
Let $V$ be a vector space and $\Omega^{\bullet}(V)$ be as above, and let $\xi:T(V^*)\to T(V^*)$ be a vector field.
\begin{enumerate}\item The {\it Lie derivative} $L_\xi:\Omega^\bullet(V)\to\Omega^\bullet(V)$ with respect to $\xi$
is defined by:
$$L_\xi(x):=\xi(x),\quad L_\xi(dx):=(-1)^{|\xi|}d(\xi(x)),\quad\mbox{for}\quad x\in T(V^*).$$
\item The {\it contraction} $\iota_\xi:\Omega^\bullet(V)\to\Omega^\bullet(V)$ of $\xi$ is defined by:
$$\iota_\xi(x):=0,\quad \iota_\xi(dx):=\xi(x),\quad\mbox{for}\quad x\in T(V^*).$$
\end{enumerate}
The Lie derivative and contraction preserve the commutators and hence descend to $DR^\bullet(V)$.
\end{definition}

 \begin{definition}[Symplectic form and symplectic field]Let $V$ be a vector space. A $2$-form $\omega\in DR^2(V)$
 is called {\it symplectic} if
 \begin{enumerate}
 \item $\omega$ is closed, i.e. $d\omega=0$;
 \item $\omega$ is non-degenerate, i.e. there is a bijection:
 $$\begin{array}{ccl}
 Der(T(V^*))&\longrightarrow& DR^1(V),\\
 \xi&\longmapsto&\iota_\xi(\omega).
 \end{array}$$
 \end{enumerate}
 A vector field $\xi:T(V^*)\to T(V^*)$ is called a {\it symplectic vector field} if $L_\xi(\omega)=0$.
 \end{definition}

A symplectic form $\omega\in DR^2(V)$ is called a {\it constant form} if it can be written as $\displaystyle\sum_i dx_i\wedge dy_i$ for some
functions $x_i,y_i\in V^*$. The following theorem is due to Kontsevich and others cited above:

\begin{theorem}[Lie bialgebra on the noncommutative $0$-forms]\label{poisson_liebi}
Let $V$ be a vector space with a constant symplectic form $\omega$. The symplectic form on $V$ induces
a symplectic form on $V^*$ via the identity $V\to V^*: v\mapsto \omega(v,\cdot)$, which is still denoted by $\omega$.
Define $\{\,,\,\}:DR^0(V)\otimes DR^0(V)\to DR^0(V)$ and $\Delta: DR^0(V)\to DR^0(V)\otimes DR^0(V)$ as follows:
for $a_1,a_2,\cdots, a_n;b_1,b_2,\cdots,b_m\in V^*$,
\begin{equation*}
\{[a_1a_2\cdots a_n],[b_1b_2\cdots b_n]\}=\sum_i\sum_j\pm\omega(a_i,b_j)([a_1\cdots a_{i-1}b_{j+1}\cdots  b_nb_1\cdots b_{j-1}a_{i+1}\cdots a_n]),
\end{equation*}
and
\begin{eqnarray*}
&&\Delta([a_1a_2\cdots a_n])\\
&=&\frac{1}{2}\sum_{i<j}\pm\omega(a_i,a_j)([a_1\cdots a_{i-1}a_{j+1}\cdots a_n])\otimes([a_{i+1}\cdots a_{j-1}])\\
&+&\frac{1}{2}\sum_{i<j}\pm\omega(a_i,a_j) ([a_{i+1}\cdots a_{j-1}])\otimes([a_1\cdots a_{i-1}a_{j+1}\cdots a_n]).
\end{eqnarray*}
Then $(DR^0(V), \{\,,\,\},\Delta)$ forms an involutive Lie bialgebra.
\end{theorem}

In terms of the terminologies introduced above, the Lie bracket is
identical to the Poisson bracket on the noncommutative $0$-forms.
The reader may refer to the above listed references for more
details.

So far in this subsection we have assumed only that $V$ is a vector space. Now suppose $V$ is a cyclic $A_\infty$ space. Let $g:V\otimes V$
be the non-degenerate pairing on $V$, which induces a symplectic form $\omega$ on $\Sigma V$.
There is an isomorphism (Quillen \cite[Lemma 1.2, pp. 144]{Quillen})
$$\Cycl^*(V)=\mathrm{ker}(1-T)\cong\mathrm{coker}(1-T),$$
namely, the cyclic cochain complex of $V$ is isomorphic to the quotient commutator space.
With this identity, the Lie bialgebra defined in the cyclic cochain complex is exactly the one for $(\Sigma V,\omega)$ defined above. Thus we obtain:

\begin{proposition}Let $(V,\{m_i\},g)$ be a cyclic $A_\infty$ algebra over $k$, and denote by $\omega$ the the induced symplectic form on $\Sigma V$.
Then cyclic cochain complex $\Cycl^*(V)$ is isomorphic
to the noncommutative $0$-forms on $\Sigma V$ as Lie bialgebras, where
the Lie bialgebra of the former is given in Theorem~\ref{thm_liebi}
and the Lie bialgebra of the latter is given in Theorem~\ref{poisson_liebi}.
\end{proposition}

\subsection{Example from string topology}\label{ex_st}
Let $M$ be a compact, smooth manifold and $LM$ its free loop space.
In \cite{CS02} (see also \cite{SullivanSurvey}) Chas and Sullivan
showed, under the program of string topology, that the equivariant
homology of the free loop space $LM$ has an involutive Lie bialgebra
structure. Such a Lie bialgebra even holds on the chain level, which
can be roughly described as follows: Suppose $\a,\b$ are two
transversal equivariant chains in $LM$ (that means, they are in
general position); one tries to move the base points of $\a$ and
$\b$ along themselves. At time $t_1$ for $\a$ and $t_2$ for $\b$,
one forms a new chain whose base points are the intersections of
these two base points at $t_1$ and $t_2$, and whose fiber is the
concatenation of the loops from $\a$ and $\b$. By the $S^1$
invariance of $\a$ and $\b$, this new chain is $S^1$ invariant, too.
The Lie bracket of $\a$ and $\b$ (called the {\it string bracket})
is thus defined by considering all $t_1$ and $t_2$ in $S^1$. The Lie
cobracket of an equivariant chain, say $\a$, is defined as follows:
again, assume $\a$ is in general position, one moves the base points
of $\a$ along itself. Suppose at time $t_1$ and $t_2$, the base
points intersect, then one forms two new chains in $LM$ whose base
points are the intersections, and whose fibers are the arcs in $\a$
from $t_1$ to $t_2$ and the rest (note that both arcs are now
circles), respectively. Chas and Sullivan argued that with the Lie
bracket and cobracket defined this way, the equivariant chain
complex of $LM$ forms an involutive Lie bialgebra over $\mathbb Z$.

There are some difficulties in the understanding of
the Lie bialgebra of string topology, especially in its definition the partially defined transversal intersection
is used. However, we may still get some idea from the constructions in the previous sections.
First, with the work of K.-T. Chen (\cite{Chen77}) and J. D. S. Jones (\cite{Jo87}),
it is nowadays known to topologists that:

\begin{theorem}[Chen; Jones]Suppose $M$ is a simply connected manifold and $C^*(M)$ is a cochain model
of $M$, then the Hochschild (resp. cyclic) chain complex of $C^*(M)$ is qusi-isomorphic to the ordinary (resp. equivariant) chain complex of $LM$.
\end{theorem}

As a corollary, the cyclic {\it cochain} complex of $C^*(M)$ is then quasi-isomorphic to the equivariant {\it chain} complex of $LM$.

Recently, Hamilton and Lazarev constructed in \cite[Section 5]{HL} a finite dimensional cyclic $C_\infty$ algebra over $\mathbb Q$,
modeling the cochain complex of $M$. A cyclic $C_\infty$ algebra is a homotopy commutative associative algebra with an invariant bilinear form, it is
automatically a cyclic $A_{\infty}$ algebra.

\begin{theorem}\label{thm_st}
Let $M$ be a simply connected manifold and $V$ be the cyclic $A_\infty$ algebra constructed by Hamilton and Lazarev that models the cochain
complex of $M$. Then the Lie bialgebra of $V$ given in Theorem~\ref{main_thm} models the one of string topology tensored with the rational numbers.
\end{theorem}

\begin{proof}
First observe that the Hochschild and cyclic (co)homology groups are invariant
under quasi-isomorphic $A_\infty$ algebras (see for example~\cite[Lemma 5.3]{GJ}).
There is an isomorphism on the cyclic cohomology groups of $C^*(M)$ and the $A_\infty$ algebra $V$ of
Hamilton and Lazarev.

The operator $N:T(sV^*)\to T(sV^*)$ exactly models the base points rotation of Chas and Sullivan, and the Lie bracket and cobracket given
in Section 3 then models the concatenation and splitting of loops over the intersection points.
\end{proof}

\begin{remark}A similar construction of the Lie bialgebra
using the cochain model of Lambrechts and Stanley (\cite{LS}) which is a strict DG commutative algebra with a non-degenerate pairing
has been obtained by the author with Eshmatov and Gan (\cite{CEG}).
\end{remark}

\section{Application in symplectic topology}

In this section we give an example of the Lie bialgebra of Calabi-Yau $A_\infty$ categories
that arise from symplectic topology. The details are quite involved, so in the following we will
only sketch the main idea.

We start with the definition of the Maslov class of
Lagrangian submanifolds.
Let $(\mathbb R^{2n},\omega)$ be the standard symplectic vector space. Denote
by $\Lag_n=\Lag(\mathbb R^{2n},\omega)$ the set of all linear
Lagrangian subspaces. It is known that $\pi_1(\Lag_n)\cong\mathbb Z$.

Suppose $M$ is an exact symplectic manifold, and $L$ is an exact
Lagrangian submanifold (see Example~\ref{ex_fs} in \S2.1). Let
$\phi:(D^2,\partial D^2)\to (M,L)$ be a map representing
$\pi_2(M,L)$. For each $t\in\partial D^2$ we have a Lagrangian
subspace $T_{\phi(t)}L\subset T_{\phi(t)}M$, which gives a map
$S^1\to\Lag_n$. It determines an element in
$\pi_1(\Lag_n)\cong\mathbb Z$, which is called the {\it Maslov
index} of $[\phi]$, and is denoted by $\mu([\phi])$. The map
$[\phi]\mapsto\mu([\phi])$ is called the {\it Maslov class} of $L$.

In the following we assume all the exact Lagrangian submanifolds
have vanishing Maslov classes. It is proved by symplectic topologists (see, for example, \cite[Theorem 1.1]{Fuk09} and \cite[\S8]{Seidel1})
that morphism spaces of the
Fukaya category of Lagrangian submanifolds are vector spaces over
$k=\mathbb Q$ or $\mathbb R$, and are $\mathbb Z$ graded.
Moreover, there is a pairing of degree $n$ on the morphism spaces
$$\Hom(A,B)\otimes\Hom(B,A)\to\Hom(A,A)\stackrel{\mathrm{tr}}{\to} k$$
induced from the Poicar\'e duality, making the Fukaya category into
a Calabi-Yau $A_\infty$ category.
As a corollary to Theorem~\ref{thm_liebi}, we have the following:

\begin{theorem}
Let $M$ be an exact symplectic $2n$-manifold with contact type boundary. The cyclic cohomology
of the Fukaya category $\mathcal Fuk(M)$ of exact Lagrangian submanifolds has the structure of an involutive Lie bialgebra.
\end{theorem}

\subsection{Symplectic field theory}
Symplectic field theory is a program initiated by Eliashberg, Givental and Hofer more than ten years ago
(\cite{EGH}). It has been widely studied by mathematicians and has
obtained many interesting results in symplectic and contact topology.
In the following we will touch only a small part which is necessary for our discussion.

Let $(W,\eta)$ be a co-oriented contact manifold, where $\eta$ is the contact form. There are two concepts associated to $W$:
\begin{itemize}\item The {\it symplectization} of $W$ is, by definition, $W\times\mathbb R$ with symplectic form
$d(e^s\eta)$, where $s$ is the coordinate of the factor $\mathbb R$.
\item A {\it Reeb orbit} in $W$ is a closed orbit of the Reeb vector field $Y$:
$$\eta(Y)=1,\quad\iota(Y)d\eta=0.$$
If the contact form $\eta$ is
generic, then the set of Reeb orbits is discrete. To each Reeb
orbit, one may assign an integer, called the {\it Conley-Zehnder index} of the orbit,
which is a version of Maslov index (\cite[\S1.2]{EGH}).
\end{itemize}
Choose a generic contact form and consider a $J$-holomorphic curve $\phi:\Sigma\to W\times\mathbb R$. A theorem of Hofer says that
if $\phi$ has non-removable singular points, then these singular points can only approach the Reeb orbits in $W$
at $\pm\infty$.
Moreover, since the Reeb orbits are discrete, one may consider
the moduli space of $J$-holomorphic curves with a given type. Namely, let
$\mathcal M(\Sigma; q_1^+,q_2^+,\cdots,q_m^+;q_1^-,q_2^-,\cdots,q_n^-)$ be the
set of $J$-holomorphic curves $\Sigma\to W\times\mathbb R$ such that the only singular points of $\Sigma$
are $q_1^+,q_2^+,\cdots,q_m^+$ at $+\infty$ and $q_1^-,q_2^-,\cdots,q_n^-$ at $-\infty$.
It is proved in symplectic field theory that the union of $\mathcal M(\cdots)$'s with all singular type
forms a stratified space, whose codimension great than zero
strata of a given singular type is described by the ``broken" curves, which, topologically,
can be realized as $J$-holomorphic curves which are stretched
to be infinitely long in the middle of $W\times\mathbb R$ at some time $s$.
$\mathcal M(\cdots)$ always admits an $\mathbb R$ action,
which shifts the $J$-holomorphic curves along the $\mathbb R$
factor in $W\times\mathbb R$.

Now suppose $M$ is a symplectic manifold with contact type boundary $W$.
The {\it linearized contact homology} of $W$ is defined as follows (for more details see \cite{CL}):
let $\Cont_*^{\mathrm{lin}}(W)$ be the linear space spanned by the Reeb orbits (technically, they should be {\it good} in the sense
of \cite{EGH}), graded by the Conley-Zehnder indices.
Define a linear operator
$$d:\Cont_*^{\mathrm{lin}}(W)\to \Cont_{*-1}^{\mathrm{lin}}(W)$$
by
$$d(q^+)=\sum_{q^-:|q^-|=|q^+|-1}\#(\mathcal M(\mbox{cylinder}; q^+;q^-)/\mathbb R)\cdot q^-.$$
The compactness of $\mathcal M(\cdots)$ makes $d$ into a differential.
The following is achieved in symplectic field theory by Cieliebak and Latschev (\cite[Theorem A]{CL}):

\begin{theorem}[Cieliebak and Latschev]
Let $M$ be a symplectic manifold with contact type boundary $W$. The chain complex $\Cont_*^{\mathrm{lin}}(W)$ has the structure
of a bi-Lie$_\infty$ algebra, which, induces on the linearized contact homology of $W$ an
involutive Lie bialgebra of degree $(-1,5-2n)$.
\end{theorem}

A bi-Lie$_\infty$ algebra is a homotopy version of an involutive Lie
bialgebra, where the Jacobi identity, co-Jacobi identity, Drinfeld
compatibility and involutivity all hold up to homotopy (we shall not
go to the details of bi-Lie$_\infty$ algebras; for some more
discussions, see \cite{CL}). Basically, the Lie bracket is given by
counting the $J$-holomorphic spheres with two non-removable singular
points at $+\infty$ and one non-removable singular point at
$-\infty$ in $W\times\mathbb R$. Namely, suppose $q_1,q_2$ are two
Reeb orbits in $W$. Then the Lie bracket
$$[q_1,q_2]:=\sum_{q}\#(\mathcal M(\mbox{sphere}\backslash\mbox{\{3 points\}}; q_1^+,q_2^+;q^{-})/\mathbb R)\cdot q,$$
where $\mathcal M(\mbox{sphere}\backslash\mbox{\{3 points\}}; q_1^+,q_2^+;q^{-})$ is the moduli space
of $J$-holomorphic spheres with 3 non-removable singular points, two of which approach $q_1$ and $q_2$ in $W$ at $+\infty$ and
the third one approaches $q$ in $W$ at $-\infty$.
By considering the $J$-holomorphic spheres with 3 non-removable singular points at $+\infty$ and 1 non-removable singular point
at $-\infty$ in $W$, one obtains that the Jacobi identity for the Lie bracket holds up to homotopy.

In a similar fashion one may define the Lie cobracket and prove that the co-Jacobi identity holds up to homotopy, too. The involutivity
and Drinfeld compatibility hold up tomotopy due to the study of the moduli space of $J$-holomorphic tori with 2 non-removable singular points
and $J$-holomorphic spheres with 4 non-removable singular points (2 at $+\infty$ and 2 at $-\infty$) respectively.

In the theorem of Cieliebak-Latschev, if we shift the degrees of the Reeb orbits up by $n-3$, then the Lie bialgebra on the
linearized contact homology has degree $(2-n,2-n)$. In the following we take this degree shifting.

\subsection{From Fukaya category to symplectic field theory}

Let $M$ be an exact symplectic manifold with contact type boundary $W$.
$M$ has a {\it symplectic completion}, which is $$M\cup_{id:W\to W\times\{0\}}(W\times\mathbb R^{\ge 0}),$$ and is denoted by $\tilde M$.
$\tilde M$ is a symplectic manifold whose symplectic form is induced from $M$ and $W\times\mathbb R^{\ge 0}$.

\begin{theorem}\label{thm_liebimap}
Let $M$ be a symplectic manifold with contact type boundary $W$, and
let $\tilde M$ be the symplectic completion.
Let
$$f:\Cont_*^{\mathrm{lin}}(W)\to\Cycl^*(\mathcal Fuk(M))$$
be the map such that the value of
$f(q)$, $q\in\Cont_*(W)$, at $(a_1,a_2,\cdots,a_n)\in\Hom(A_1,A_2)\otimes\Hom(A_2,A_3)\otimes\cdots\otimes\Hom(A_n,A_1)$ is the cardinality
of $J$-holomorphic disks whose boundary lies in $A_1\cup A_2\cup \cdots\cup A_n$ and goes through $a_1,a_2,\cdots,a_n$ in cyclic order
and whose non-removable singular point in $\tilde M$ at
$+\infty$ approaches $q$.
Then $f$
is a map of chain complexes and induces on the homology a map of Lie bialgebras.
\end{theorem}

\begin{proof}[Sketch of proof]First, we show $f$ is a chain map.
For any $q\in\Cont_*(W)$, and $(a_1,a_2,\cdots,a_n)\in\Hom(A_1,A_2)\otimes\Hom(A_2,A_3)\otimes\cdots\otimes\Hom(A_n,A_1)$,
consider the moduli space of $J$-holomorphic disks with boundaries in $A_1\cup A_2\cup\cdots\cup A_n$ and through $a_1,a_2,\cdots,a_n$
in cyclic order and with a non-removable singular point approaching $q$.
Suppose the moduli space is 1-dimensional, then its boundary is composed of the following two types of broken $J$-holomorphic curves:
\begin{itemize}
\item The first type is a broken curve
whose first piece is a $J$-holomorphic cylinder in $W\times\mathbb R$ with singular points approaching $q$ at $+\mathbb R$ and $q'$ at
$-\infty$, and whose second piece is a $J$-holomorphic disk with its boundary lying in $A_1\cup A_2\cup\cdots\cup A_n$ as before and with
a non-removable singular point approaching $q'$ in $\tilde M$. In other words, it comes from stretching the original singular
$J$-holomorphic disk in
$\mathbb R^+$ direction at some middle point in $\tilde M$;
\item The second type is a broken curve
whose first piece is a $J$-holomorphic disk with its boundary lying in $A_1\cup\cdots\cup A_{i}\cup A_{j}\cup\cdots\cup A_n$ and going
through $a_1,\cdots, a_{i-1}, p,a_j,\cdots, a_n$ in cyclic order
and with a non-removable singular point approaching $q$ at $+\infty$, and whose second piece is a $J$-holomorphic curve with
its boundary lying in $A_i\cup A_{i+1}\cup\cdots\cup A_j$ and going through $p,a_i,a_{i+1},\cdots, a_{j-1}$, where
$p\in A_i\cap A_j$.
In other words, it comes from the bubbling-off (of disks) from the original singular $J$-holomorphic disk.
\end{itemize}
Now to show $f(dq)=b(f(q))$, consider their values at
$(a_1,a_2,\cdots,a_n)\in\Hom(A_1,A_2)\otimes\Hom(A_2,A_3)\otimes\cdots\otimes\Hom(A_n,A_1)$.
Note $f(dq)(a_1,a_2,\cdots,a_n)$ is exactly the number of broken $J$-holomorphic curves of the first type
and $b(f(q))(a_1,a_2,\cdots,a_n)=f(q)(b(a_1,a_2,\cdots,a_n))$ is exactly the number of broken $J$-holomorphic
curves of the second type.
It follows that their cardinalities are equal.

Next we show that $f$ maps the bracket to the bracket up to homotopy, which induces a Lie algebra map on the homology.
For $q_1,q_2\in\Cont_*^{\mathrm{lin}}(W)$ and
$(a_1,a_2,\cdots,a_n)\in\Hom(A_1,A_2)\otimes\Hom(A_2,A_3)\otimes\cdots\otimes\Hom(A_n,A_1)$, we
now consider $J$-holomorphic disks with their boundary
in $A_i\cup A_2\cup\cdots\cup A_n$ and through $a_1,a_2,\cdots,a_n$ in cyclic order,
and having
two non-removable singular points approaching $q_1,q_2$ in $W$ at $+\infty$.
Suppose the moduli space is 1-dimensional, then its boundary points are composed of four types of broken $J$-holomorphic
curves (imaging a pair of pants are placed vertically with two sleeves at $+\infty$ and one at $0$):
\begin{itemize}
\item The broken point of the pants occurs at one of the upper sleeves but not the other; in other words, one
of them is stretched to be infinitely long.
\item The broken point of the pants occurs at the bottom sleeve; in other words, there is a bubbling off of a disk at the bottom sleeve.
\item The broken point of the pants occurs at somewhere between the saddle point and the bottom sleeve;
\item The broken point occurs when the saddle point is pushed down to $0$.
\end{itemize}
These four cases corresponds to the following operations respectively:
(1) $f_2(dq_1,q_2)$ and $f_2(q_1,dq_2)$; (2) $b(f_2[q_1,q_2])$; (3) $f[q_1,q_2]$; and (4) $[f(q_1),f(q_2)]$.
Here by $f_2(-,-)$ we mean counting the $J$-holomorphic disks with two non-removable singular points.
It follows then that
$$f[q_1,q_2]-[f(q_1),f(q_2)]=f_2(dq_1,q_2)+f_2(q_1, dq_2)-bf_2(q_1,q_2),$$
that is, $f$ is a Lie algebra map up to homotopy.

Finally we show that $f$ maps the cobracket to the cobracket up to homotopy. By considering
$J$-holomorphic cyclinders with both boundaries fixed on two sets of cyclic chain complex in $\mathcal Fuk(M)$ and with
1 non-removable singular point approaching a Reeb orbit in $W$ at $+\mathbb R$ (or equivalently, the $J$-holomorphic
pants with two sleeves fixed in the Lagrangian submanifolds and one sleeve going to $+\infty$), with
the similar argument as above, one sees that $f$ is a Lie coalgebra map up to homotopy.
\end{proof}

From the proof, one sees that if we consider all $J$-holomorphic curves with several non-removable
singular points and several boundaries, then $f$ is in fact a map of bi-Lie$_\infty$ algebras.
It is interesting then to ask when $f$ induces an isomorphism of Lie bialgebras on the homology level.
We hope to address this problem somewhere else.

\end{document}